\documentclass[11pt]{article}

\usepackage[utf8]{inputenc}
\usepackage{amsmath}
\usepackage{mathtools}
\usepackage{amsfonts}
\usepackage{amssymb}
\usepackage{amsthm}
\usepackage[noabbrev]{cleveref}
\usepackage{comment}
\usepackage{algorithm}
\usepackage{algpseudocode}
\usepackage[nodisplayskipstretch]{setspace}
\usepackage{url}
\usepackage{cite}

\usepackage{graphics}
\usepackage{epsfig}
\usepackage{times}
\usepackage[margin=1in]{geometry}
\usepackage[numbers,sort,compress]{natbib}

\newtheorem{theorem}{Theorem}
\newtheorem{lemma}{Lemma}
\newtheorem{proposition}{Proposition}

\newtheorem{definition}{Definition}
\newtheorem{assumption}{Assumption}
\newtheorem{remark}{Remark}

\usepackage{xcolor}
\begin{document}

\title{Distributed and Localized Model Predictive Control. \\ Part I: Synthesis and Implementation}

\author{Carmen~Amo~Alonso \thanks{C. Amo Alonso is a graduate student in the Computing and Mathematical Sciences Department at California Institute of Technology, Pasadena, CA. {\tt\small camoalon@caltech.edu}}
\and Jing Shuang (Lisa) Li \thanks{J.S. Li is a graduate student in the Computing and Mathematical Sciences Department at California Institute of Technology, Pasadena, CA. {\tt\small jsli@caltech.edu}}
\and James~Anderson \thanks{ J. Anderson is with the Department of Electrical Engineering and the Data Science Institute at Columbia University, New York, NY. {\tt\small james.andersosn@columbia.edu}}
\and Nikolai~Matni  \thanks{N. Matni is an Assistant Professor with the Department of Electrical and Systems Engineering at the University of Pennsylvania, Philadelphia, PA.    {\tt\small nmatni@seas.upenn.edu}}}

\markboth{Submitted to IEEE Transactions on Control of Network Systems,~Vol.~\textbf{??}, No.~\textbf{??}, date\textbf{??}}
{Amo Alonso \MakeLowercase{\textit{et al.}}: Distributed and Localized Model Predictive Control: Synthesis and Implementation}

\maketitle

\begin{abstract}

The increasing presence of large-scale distributed systems highlights the need for scalable control strategies where only local communication is required. Moreover, in safety-critical systems it is imperative that such control strategies handle constraints in the presence of disturbances. In response to this need, we present the Distributed and Localized Model Predictive Control (DLMPC) algorithm for large-scale linear systems. DLMPC is a distributed \emph{closed-loop} model predictive control (MPC) scheme wherein only local state and model information needs to be exchanged between subsystems for the computation and implementation of control actions. We use the System Level Synthesis (SLS) framework to reformulate the centralized MPC problem, and show that this allows us to naturally impose localized communication constraints between sub-controllers. The structure of the resulting problem can be exploited to develop an Alternating Direction Method of Multipliers (ADMM) based algorithm that allows for distributed and localized computation of closed-loop control policies. We demonstrate that computational complexity of the subproblems solved by each subsystem in DLMPC is independent of the size of the global system. DLMPC is the first MPC algorithm that allows for the scalable computation and implementation of distributed closed-loop control policies, and deals with additive disturbances. In our companion paper, we show that this approach enjoys recursive feasibility and asymptotic stability.

\end{abstract}

\section{Introduction}

Model Predictive Control (MPC) has seen widespread success across many applications. However, the need to control increasingly large-scale, distributed, and networked systems has limited its applicability. Large-scale distributed systems are often impossible to control with a centralized controller, and even when such a centralized controller can be implemented, the high computational demand of MPC renders it impractical. Thus, efforts have been made to develop \emph{distributed} MPC (DMPC) algorithms, wherein sub-controllers solve a local optimization problem, and potentially coordinate with other sub-controllers in the network.  

\textbf{Prior work:} The majority of DMPC research has focused on cooperative setting, where sub-controllers exchange state and control action information in order to coordinate their behavior so as to optimize a global objective, typically through distributed optimization \cite{venkat_distributed_2008,zheng_networked_2013,giselsson_accelerated_2013,conte_distributed_2016,jalal_limited-communication_2017,wang_distributed_2015, venkat2005stability,sturz_distributed_2020}. Most of these approaches rely on nominal open-loop approaches, and  while nominal MPC enjoys some intrinsic robustness \cite{limonmarruedo_stability_2002}, the resulting closed-loop can be destabilized by an arbitrary small disturbance \cite{grimm_examples_2004}. Thus, to maintain robustness in the presence of additive disturbances, \emph{closed-loop} policies are desirable.

Two main closed-loop MPC approaches are used. The first approach, which we use here, is to compute \emph{dynamic} structured closed-loop policies using suitable parameterizations. This strategy was introduced in the by Goulart et al. \cite{goulart_optimization_2006}, and showed that disturbance-based parameterizations allow for distributed policies to be synthesized via convex optimization. A more recent method exploits Quadratic Invariance \cite{rotkowitz_characterization_2006} and the Youla parameterization \cite{furieri_robust_2017}. These methods allow distributed closed-loop control policies to be synthesized via convex optimization; however, the resulting optimization problem lacks structure to be amenable to distributed optimization techniques, and thus these methods do not scale to large systems. Similarly, recent works exploit the System Level Synthesis (SLS) parametrization to design robust MPC controllers \cite{sieber_system_2021,chen_robust_2020}; however, it is unclear how these can be applied in the distributed setting. The alternative approach is to extend centralized methods (i.e. constraint tightening, tube MPC). This is computationally efficient but requires pre-computed stabilizing controllers, and often relies on strong assumptions, such as the existence of a \emph{static structured} stabilizing controller (as in \cite{conte_robust_2013}) which can be NP-hard to compute \cite{blondel_np_1997}, or decoupled subsystems (as in \cite{richards_robust_2007}).

Overall, though many closed-loop formulations exist, they rely on strong assumptions and/or are unsuited for distributed computation. We seek a closed-loop DMPC algorithm that (i) computes structured feedback policies via convex optimization and (ii) can be solved at scale via distributed optimization. To the best of our knowledge, no such algorithm exists.

\textbf{Contributions:}  In this paper we address this gap and present the Distributed Localized MPC (DLMPC) algorithm for linear time-invariant systems, which allows for the distributed computation of structured feedback policies with recursive feasibility and asymptotic stability guarantees. We leverage the SLS framework \cite{wang_system_2019,wang_separable_2018,anderson_system_2019} to define a novel parameterization of distributed closed-loop MPC policies such that the resulting synthesis problem is \emph{both convex and structured}, allowing for the natural use of distributed optimization techniques. Thanks to the nature of the SLS parametrization, this approach deals with disturbances in a straightforward manner with no additional assumptions. We show that by exploiting the sparsity of the underlying distributed system and resulting closed-loop system, as well as the separability properties \cite{wang_separable_2018} of commonly used objective functions and constraints, we are able to distribute the computation via ADMM \cite{boyd_distributed_2010}, thus allowing the online computation of closed-loop MPC policies to be carried out in a scalable localized manner. Our results apply to the nominal case (as presented in \cite{amo_alonso_distributed_2020}) as well as the robust case, and provide a unifying algorithm that applies to all cases and does not rely on simplifying approximations. In the resulting implementation each sub-controller solves a low-dimensional optimization problem requiring only local communication of state and model information. Through numerical experiments, we validate these results and confirm that the complexity of the subproblems solved at each subsystem scales as $O(1)$ relative to the state dimension of the full system for both the nominal as well as the robust case.

\textbf{Paper structure:} In \S II we present the problem formulation.  In \S III, we introduce the SLS framework and show how to recast MPC as an optimization problem over system responses for nominal and robust case.  We also highlight the differences between SLS and the disturbance based feedback method proposed in \cite{goulart_optimization_2006}. In \S IV we define and analyze the DLMPC algorithm. In \S V, we present a numerical experiment and we end in \S VI with conclusions and directions of future work. In our companion paper \cite{amoalonso_guarantees_2021}, we provide theoretical guarantees for convergence, recursive feasibility and asymptotic stability of this approach, and extend the algorithm to deal with coupling through cost and constraints.

\textbf{Notation:} Lower-case and upper-case letters such as $x$ and $A$ denote vectors and matrices respectively, although lower-case letters might also be used for scalars or functions (the distinction will be apparent from the context). Bracketed indices denote time-step of the real system, i.e., the system input is $u(t)$ at time $t$, not to be confused with $x_t$ which denotes the predicted state $x$ at time $t$. Superscripted variables, e.g. $x^k$, correspond to the value of $x$ at the $k^{th}$ iteration of a given algorithm. Square bracket notation, i.e., $[x]_{i}$ denotes the components of $x$ corresponding to subsystem $i$. Calligraphic letters such as $\mathcal{S}$ denote sets, and lowercase script letters such as $\mathfrak{c}$ denote a subset of $\mathbb{Z}^{+}$, e.g. $\mathfrak{c}=\left\{1,...,n\right\}\subset\mathbb{Z}^{+}$.  Boldface lower and upper case letters such as $\mathbf{x}$ and $\mathbf{K}$ denote finite horizon signals and block lower triangular (causal) operators, respectively:
\begin{equation*} 
\mathbf{x}=\left[\begin{array}{c} x_{0}\\x_{1}\\\vdots\\x_{T}\end{array}\right], ~
K =   { {\scriptscriptstyle{\left[\begin{array}{cccc}K_{0}[0] & & & \\ K_{1}[1] & K_{1}[0] & & \\ \vdots & \ddots & \ddots & \\ K_{T}[T] & \dots & K_{T}[1] & K_{T}[0] \end{array}\right]}}},
\end{equation*}
where each $x_i$ is an $n$-dimensional vector, and each $K_{i}[j]$ represents the value of matrix $K$ at the $j^\text{th}$ time-step computed at time $i$. $\mathbf{K}(\mathfrak{r},\mathfrak{c})$ denotes the submatrix of $\mathbf{K}$ composed of $\mathfrak{r}$ rows and $\mathfrak{c}$ columns respectively. We denote the block columns of  $\mathbf K$ by $\mathbf K\{1\}$,...,$\mathbf K\{T\}$, i.e. $\mathbf K\left\{1\right\}:=[K_{0}[0]^{\mathsf{T}}\ \dots\ K_{T}[T]^{\mathsf{T}}]^{\mathsf{T}}$, and we use $:$ to indicate the range of columns, i.e. $\mathbf K\left\{2:T\right\}$ contains the block columns from the second to the last. 

\section{Problem Formulation}

Consider a discrete-time linear time invariant (LTI) system with dynamics:
\begin{equation} \label{eqn:LTV_system}
x(t+1) = Ax(t)+Bu(t)+w(t),
\end{equation}
where $x(t)\in\mathbb{R}^{n}$ is the state, $u(t)\in\mathbb{R}^{p}$ is the control input, and $w(t)\in\mathcal W \subset \mathbb{R}^{n}$ is an exogenous disturbance. The system is composed of $N$ interconnected subsystems (each having one or more states), so the state, control, and disturbance inputs can be suitably partitioned as $[x]_i$, $[u]_i$, and $[w]_i$ for each subsystem $i$, consequently inducing a compatible block structure $[A]_{ij}$, $[B]_{ij}$ in the system matrices $(A,B)$. We model the interconnection topology of the system as a time-invariant unweighted directed graph $\mathcal{G}_{(A,B)}(\mathcal E,\mathcal V)$, where each subsystem $i$ is identified with a vertex $v_{i}\in \mathcal V$ and an edge $e_{ij}\in \mathcal E$ exists whenever $[A]_{ij}\neq 0$ or $[B]_{ij}\neq 0$.

We study the case where the control input is a model predictive controller optimizing a nominal objective and is subject to constraints on the state and the input. As is standard, at each time step $\tau$ the controller solves an optimal control problem over a finite prediction horizon of length $T$ using the current state as the initial condition:
\begin{align} \label{eqn:MPC}
& \underset{{x}_{t},u_{t}, \gamma_t}{\text{min}} &  &\sum_{t=0}^{T-1}f_{t}(x_{t},u_{t})+f_{T}(x_{T})\\
& \ \text{s.t.} &  &\begin{aligned} \nonumber
    &x_{0} = x(\tau),\ x_{t+1} = Ax_{t}+Bu_{t}+w_t, \\
    &x_{T}\in\mathcal{X}_{T},\, x_{t}\in\mathcal{X}_{t},\, u_{t}\in\mathcal{U}_{t}\ \forall w_t\in\mathcal{W}_t, \\
    &u_{t} = \gamma_t(x_{0:t},u_{0:t-1}),\ t=0,...,T-1,
\end{aligned}
\end{align}
where $f_t(\cdot,\cdot)$ and $f_T(\cdot)$ are closed, proper, and convex and $\gamma_t(\cdot)$ is a measurable function of its arguments. In the nominal case, i.e., $w_t=0\ \forall t$, $\mathcal{X}_t$ and $\mathcal{U}_t$ are closed and convex sets containing the origin. When noise is present, we restrict ourselves to polytopic sets: $\mathcal{X}_t:=\{x:H_{x,t}x\leq h_{x,t}\}$ and $\mathcal{U}_t:=\{u:H_{u,t}u\leq h_{u,t}\}$, and $\mathcal{W}_t$ is considered to be norm-bounded or a polytopic set. 

Our goal is to define an algorithm that allows us to solve the MPC problem \eqref{eqn:MPC} in a distributed manner while respecting local communication constraints. To achieve this goal, we impose that information exchange - as defined by the graph $\mathcal G_{(A,B)}(\mathcal E,\mathcal V)$ - is localized to a subset of neighboring sub-controllers. In particular, we use the notion of a $d$-local information exchange constraint \cite{wang_separable_2018} to be one that restricts sub-controllers to exchange their state and control actions with neighbors at most $d$-hops away, as measured by the communication topology $\mathcal{G}_{(A,B)}$. This notion is captured by the $d$-outgoing and $d$-incoming sets of subsystem.
\begin{definition}{\label{def: in_out set}}
For a graph $\mathcal{G}(\mathcal V, \mathcal E)$, the \textit{d-outgoing set} of subsystem $i$ is $\textbf{out}_{i}(d) := \left\{v_{j}\ |\  \textbf{dist}(v_{i} \rightarrow v_{j} ) \leq d\in\mathbb{N} \right\}$. The \textit{d-incoming set} of subsystem $i$ is $\textbf{in}_{i}(d) := \left\{v_{j}\ |\ \textbf{dist}(v_{j} \rightarrow v_{i} ) \leq d\in\mathbb{N} \right\}$.  Note that $v_i \in \textbf{out}_{i}(d)\cap \textbf{in}_{i}(d)$ for all $d\geq 0$, and $ \textbf{dist}(v_{i} \rightarrow v_{j} )$ denotes the distance between $v_i$ and $v_j$ i.e., the number of edges in the shortest path connecting subsystems $i$ and $j$.
\end{definition}

Hence, we can enforce a $d$-local information exchange constraint on the MPC problem \eqref{eqn:MPC} -- where the size of the local neighborhood $d$ is a \emph{design parameter} -- by imposing that each sub-controllers policy respects 
\begin{equation} \label{eqn:comms}
\begin{aligned} 
    [u_t]_i = \gamma^i_{t}([x_{0:t}]_{j\in \textbf{in}_i(d)},[u_{0:t-1}]_{j\in \textbf{in}_i(d)}, [A]_{j,k \in \textbf{in}_i(d)}, [B]_{j,k\in  \textbf{in}_i(d)})
\end{aligned} 
\end{equation}
for all $t=0,\dots,T$ and $i=1,\dots,N$, where $\gamma_{i,t}$ is a measurable function of its arguments. This means that the closed-loop control policy at sub-controller $i$ can be computed using only states, control actions, and system models collected from $d$-hop incoming neighbors of subsystem $i$ in the communication topology $\mathcal{G}_{(A,B)}$. Given such an interconnection topology, suitable structural compatibility assumptions between the cost function and state, input, and information exchange constraints are necessary for both the synthesis and implementation of a localized control action at each subsystem.

\begin{assumption}{\label{assump: locality}}
In formulation \eqref{eqn:MPC} the objective function $f_{t}$  is such that $f_{t}(x,u)=\sum f_{t}^i([x]_{j\in\textbf{in}_{i}(d)},[u]_{j\in\textbf{in}_{i}(d)})$ for $f^i$ local functions, and the constraint sets are such that $x\in\mathcal{X}_t=\mathcal{X}_{t}^1\cap ... \cap \mathcal{X}_{t}^N$, where $x \in \mathcal{X}_t$ if and only if $[x]_{j\in\textbf{in}_{i}(d)}\in\mathcal{X}_{t}^i$ for all $i$ and $t\in\{0,...,T\}$ for $\mathcal X^i$ local sets, and idem for $\mathcal{U}_t$.
\end{assumption}

Assumption \ref{assump: locality} imposes that whenever two subsystems are coupled through either the constraints or the objective function, they then must be within the $d$-local regions -- $d$-incoming and $d$-outgoing sets -- of one another. This is a natural assumption for large structured networks where couplings between subsystems tend to occur at a local scale. We will show that under these conditions,  DLMPC allows for both a localized (a convex) synthesis and implementation of a control action at each subsystem by imposing appropriate $d$-local structural constraints on the \emph{closed-loop system responses} of the system. For the remainder of this paper we focus on developing a distributed and localized algorithmic solution and defer the design of a terminal cost and set that provide theoretical guarantees to our companion paper \cite{amoalonso_guarantees_2021}. 

\section{Localized MPC via System Level Synthesis}\label{sec:sls}

We introduce the SLS framework \cite{wang_system_2019,wang_separable_2018,anderson_system_2019} and justify its utility in MPC problems. We show how SLS naturally allows for locality constraints \cite{wang_separable_2018} to be imposed on the system responses and corresponding controller implementation, and discuss how state and input constraints can be imposed in the presence of disturbances extending the results from \cite{chen_system_2019}, leading to the formulation of the Distributed and Localized MPC problem.

\subsection{Time domain System Level Synthesis}

The following is adapted from $\S2$ of \cite{anderson_system_2019}.  Consider the dynamics of system \eqref{eqn:LTV_system} and let $u_t$ be a causal linear time-varying state-feedback controller, i.e., $u_t=K_t(x_0,x_1,...,x_t)$ where $K_t$ is some linear map to be designed.\footnote{Our assumption of a linear policy is without loss of generality, as an affine control policy $u_t = K_t(x_{0:t}) + v_t$ can always be written as a linear policy acting on the homogenized state $\tilde{x} = [x;1]$.} Let $Z$ be the block-downshift matrix,\footnote{A matrix with identity matrices along its first block sub-diagonal and zeros elsewhere.} and define $\hat A:=\mathrm{blkdiag}(A,...,A)$ and $\hat B:=\mathrm{blkdiag}(B,...,B,0)$. Using the signal (bold) notation, we can compactly write the closed-loop behavior of system \eqref{eqn:LTV_system} under the feedback law $\mathbf u = \mathbf{K} \mathbf x$, over the horizon $t=0,...,T$, which can be entirely characterized by {the system responses $\mathbf \Phi_x$ and $\mathbf \Phi_u$}:
\begin{equation} \label{eqn:Phis}
\begin{split}
\mathbf{x} & = (I-Z(\hat A+\hat B\mathbf{K}))^{-1}\mathbf{w} =: \mathbf\Phi_x \mathbf w\\
\mathbf{u} & = \mathbf{K}(I-Z(\hat A+\hat B\mathbf{K}))^{-1}\mathbf{w} =: \mathbf\Phi_u \mathbf w.
\end{split}
\end{equation}
Here $\mathbf{x}$, $\mathbf{u}$ and $\mathbf{w}$ are the finite horizon signals corresponding to state, control input, and disturbance respectively. By convention, we define the disturbance to contain the initial condition, i.e., $\mathbf{w} = [x_{0}^\mathsf{T}\ w_{0}^\mathsf{T}\ \dots \ w_{T-1}^\mathsf{T}]^\mathsf{T}$.

The approach taken by SLS is to directly parameterize and optimize over the set of achievable system responses $\{\mathbf{\Phi}_x,\mathbf{\Phi}_u\}$ from the exogenous disturbance $\mathbf{w}$ to the state $\mathbf{x}$ and the control input $\mathbf{u}$, respectively.
\begin{theorem}{\label{thm: SLS}} \emph{(Theorem 2.1 of \cite{anderson_system_2019})}
For the system  \eqref{eqn:LTV_system} evolving under the state-feedback policy $\mathbf u = \mathbf K \mathbf x$, for $\mathbf{K}$ a block-lower-triangular matrix, the following are true
\begin{enumerate}
    \item The affine subspace
    \begin{equation}\label{eqn:Z_AB}
        Z_{AB}\mathbf{\Phi}:=\left[I-Z\hat A\ \ -Z\hat B\right]\left[\begin{array}{c}\mathbf{\Phi}_{x}\\\mathbf{\Phi}_{u}\end{array}\right] = I
    \end{equation}
    with lower-triangular $\{\mathbf{\Phi}_{x},\mathbf{\Phi}_{u}\}$ parameterizes all possible system responses \eqref{eqn:Phis}.
    
    \item For any block lower-triangular matrices $\left\{\mathbf{\Phi}_{x},\mathbf{\Phi}_{u}\right\}$ satisfying \eqref{eqn:Z_AB}, the controller $\mathbf{K} = \mathbf{\Phi}_{u}\mathbf{\Phi}_{x}^{-1}$ achieves the desired response \eqref{eqn:Phis}.
\end{enumerate}
\end{theorem}

The SLS framework relies on part $1$ of Theorem \ref{thm: SLS} to reformulate optimal control problems as a search over system responses $\mathbf \Phi$ lying in subspace \eqref{eqn:Z_AB}, rather than an optimization problem over states and inputs $\{\mathbf x, \mathbf u\}$. Using parametrization \eqref{eqn:Phis}, we reformulate the MPC subroutine \eqref{eqn:MPC} in terms of the system responses as
\begin{align}
& \underset{\mathbf{\Phi}}{\text{min}} & f(\mathbf{\Phi}\{1\}x_{0})  \label{eqn:MPC-SLS} & ~\\
&  \text{s.t.} & Z_{AB}\mathbf{\Phi}=I, & ~ x_0 = x(\tau), ~ \mathbf{\Phi}\mathbf w\in \mathcal{P}\ \forall\mathbf w \in\mathcal W, \nonumber
\end{align}
where the polytopes $\mathcal{W} := \otimes_{t=0}^{T} \mathcal{W}_t$ and the polytope $\mathcal P$ is defined so that $\mathbf{\Phi}\mathbf{w}\in\mathcal{P}$ if and only if $x_{t}\in\mathcal{X}_{t},\text{ and } u_{t}\in\mathcal{U}_{t}$, for all $t=0,...,T-1$. Notice that $\mathcal W$ is defined so that it does not restrict $x_0$. The objective function $f$ is defined consistent with objective function of problem \eqref{eqn:MPC}. Note that Assumption \ref{assump: locality} directly translates to the objective function and constraint set of the SLS reformulation \eqref{eqn:MPC-SLS}. We emphasize that $\mathbf \Phi\{1\} x_0$ appears in the objective function as it corresponds to the nominal (noise-free) state and input responses. 

The equivalence between the the MPC SLS problem \eqref{eqn:MPC-SLS} and the original MPC problem \eqref{eqn:MPC} stems from the well known fact - restated in terms of the SLS parameterization - that linear time-varying controllers are as expressive as nonlinear controllers over a finite horizon, given a fixed initial condition and noise realization. In fact, for a fixed initial condition and noise realization $\mathbf{w}$, any control sequence $\boldsymbol{u}(\mathbf{w}) := [u_0^\top,\dots,u_{T-1}^\top]^\top$ can be achieved by a suitable choice of feedback matrix $\boldsymbol K(\mathbf{w})$ such that $\boldsymbol{u}(\mathbf{w}) = \boldsymbol K(\mathbf{w})\mathbf{x}$ (that such a matrix always exists follows from a dimension counting argument).  As this control action can be achieved by a linear-time-varying controller $\boldsymbol K(\mathbf{w})$, Theorem \ref{thm: SLS} states that there exists a corresponding achievable system response pair $\left\{\mathbf{\Phi}_{x},\mathbf{\Phi}_{u}\right\}$ such that $\boldsymbol{u}(\mathbf{w}) = \boldsymbol \Phi_u\mathbf{w}$. Thus the SLS reformulation introduces no conservatism relative to open-loop MPC. We discuss the closed-loop setting at the end of this section, and show that the disturbance based parametrization \cite{goulart_optimization_2006} is a special case of ours. 

\textbf{Why use SLS for Distributed MPC:} In the centralized setting, where both the system matrices $(A,B)$ and the system responses $\left\{\mathbf{\Phi}_{x},\mathbf{\Phi}_{u}\right\}$ are dense without information constraints, the SLS parameterized problem \eqref{eqn:MPC-SLS} is slightly more computationally costly than the original MPC problem \eqref{eqn:MPC}, as there are now $n(n+p)T$ decision variables, as opposed to $(n+p)T$ decision variables. However, under suitable localized structural assumptions on $f_t$, and constraint sets $\mathcal X_t$, and $\mathcal U_t$, lifting to this higher dimensional parameterization makes the problem becomes decomposable since it allows to take advantage of the the structure of the underlying system. This allows for not only the convex synthesis of a distributed closed-loop control policy (as is similarly done in \cite{goulart_optimization_2006,furieri_robust_2017}), but also for the solution of this convex synthesis problem to be computed using distributed optimization. 

This latter feature is one of the main contributions of this paper, and in particular, we show that the resulting number of optimization variables in the local subproblems solved at each sub-system scales as $O(d^2T)$, where $d$ is the size of the neighborhood for each subsystem as per Definition \ref{def: in_out set} (usually $d\ll n$) and $T$ is the time horizon, hence independent of the global system size $n$. To the best of our knowledge, this is the first distributed closed-loop MPC algorithm with such properties.

\subsection{Locality in System Level Synthesis} 

Here we illustrate how to enforce the information sharing constraint \eqref{eqn:comms} in the SLS framework, and how localized system responses result in a localized controller implementation. 

A key advantage of using the SLS framework is that the system responses not only parametrize the closed-loop map but also provide a controller realization. In particular, the controller achieving the system responses \eqref{eqn:Phis} can be implemented as
\begin{equation}\label{eqn:implementation}
\begin{aligned}
        \mathbf{u}=\mathbf{\Phi}_u\mathbf{\hat{w}},\ \ \  \mathbf{\hat{x}}=(I - \mathbf{\Phi}_x)\mathbf{\hat{w}},\ \ \ 
        \mathbf{\hat{w}}=\mathbf{x}-\mathbf{\hat{x}},
\end{aligned}
\end{equation}
where $\mathbf{\hat{x}}$ is the nominal state trajectory, and $\mathbf{\hat{w}}=Z\mathbf{w}$ is a delayed reconstruction of the disturbance. The advantage of this controller implementation, as opposed to  $\mathbf u = \mathbf \Phi_u \mathbf \Phi_x^{-1} \mathbf x$, is that any structure imposed on the system response $\{\mathbf\Phi_u, \mathbf \Phi_x\}$ translates directly to structure on the controller implementation \eqref{eqn:implementation}. This is particularly relevant for imposing locality constraints, and we will show how locality in system responses translates into locality of the controller implementation.

We begin by defining the notion of $d$-localized system responses, which follows naturally from the notion of $d$-local information exchange constraints \eqref{eqn:comms}. They consist of system responses with suitable sparsity patterns such that the information exchange needed between subsystems to implement the controller realization \eqref{eqn:implementation} is limited to $d$-hop incoming and outgoing neighbors. 

\begin{definition}{\label{def: locality}}
Let $[\mathbf{\Phi}_{x}]_{ij}$ be the submatrix of system response $\mathbf{\Phi}_x$ describing the map from disturbance $[w]_{j}$ to the state $[x]_i$ of subsystem $i$. The map $\mathbf{\Phi}_{x}$ is \textit{d-localized} if and only if for every subsystem $j$, $[\mathbf{\Phi}_{x}]_{ij}=0\ \forall\ i\not\in\textbf{out}_{j}(d)$. The definition for \textit{d-localized} $\mathbf{\Phi}_u$ is analogous but with disturbance to control action $[u]_i$.
\end{definition}

When the system responses are $d$-localized, then so is the controller implementation \eqref{eqn:implementation}. In particular, by enforcing $d$-localized structure on $\mathbf{\Phi}_{x}$, only a local subset $[\hat{\mathbf{w}}]_{j\in\textbf{in}_i(d)}$ of $\hat{\mathbf{w}}$ are necessary for subsystem $i$ to compute its local disturbance estimate $[\hat{\mathbf{w}}]_{i}$, which ultimately means that only local communication is required to reconstruct the relevant disturbances for each subsystem. Similarly, if $d$-localized structure is imposed on $\mathbf{\Phi}_{u}$, then only a local subset $[\hat{\mathbf{w}}]_{j\in\textbf{in}_i(d)}$ of the estimated disturbances $\hat{\mathbf{w}}$ is needed for each subsystem to compute its control action $[\mathbf u]_i$. Hence, each subsystem only needs to collect information from its $d$-incoming set to implement the control law \eqref{eqn:implementation}, and it only needs to share information with its $d$-outgoing set to allow for other subsystems to implement their respective control laws.
Furthermore, such locality constraints are enforced via an affine subspace constraint in the SLS formulation \eqref{eqn:MPC-SLS}.
\begin{definition}{\label{def: locality constraints}}
A subspace $\mathcal{L}_d$ enforces a $d$\textit{-locality constraint} if $(\mathbf{\Phi}_{x},\mathbf{\Phi}_{u})\in\mathcal{L}_d$ implies that $\mathbf{\Phi}_{x}$ is $d$-localized and $\mathbf{\Phi}_{u}$ is $(d+1)$-localized. \footnote{Notice that we are imposing $\mathbf{\Phi}_{u}$ to be $(d+1)$-localized because in order to localize the effects of a disturbance within the region of size $d$, the ``boundary" controllers at distance $d+1$ must take action (for more details the reader is referred to \cite{anderson_system_2019}).} A system $(A,B)$ is then $d$-localizable if the intersection of $\mathcal{L}_d$ with the affine space of achievable system responses \eqref{eqn:Z_AB} is non-empty.
\end{definition}

 \begin{remark} Although $d$-locality constraints are always convex subspace constraints, not all systems are $d$-localizable. The locality diameter $d$ can be viewed as a design parameter, and for the remainder of the paper, we assume that there exists a $d\ll n$ such that the system $(A,B)$ to be controlled is $d$-localizable. Notice that the parameter $d$ is tuned independently of the horizon $T$, and captures how ``far'' in the interconnection topology a disturbance striking a subsystem is allowed to spread -- as described in detail in \cite{wang_separable_2018}, localized control is a spatio-temporal generalization of deadbeat control.
 \end{remark}

While it was not possible to incorporate locality \eqref{eqn:comms} into the classical MPC formulation \eqref{eqn:MPC} in a convex and computationally tractable manner, it is straightforward to do so via the affine constraint $(\mathbf{\Phi_x,\mathbf\Phi_u)\in\mathcal L_d}$ as per Definition \ref{def: locality}, with the only requirement that some mild compatibility assumptions as per Assumption \ref{assump: locality} between the cost functions, state and input constraints, and $d$-local information exchange constraints are satisfied. In the remainder of this section we exploit this fact when tackling robust state and input constraints, which in turn results in a problem structure that is amenable to using distributed optimization for computing the resulting control policies using only local information.

 \subsection{State and input constraints in System Level Synthesis}

Here we extend the method to deal with robust state and input constraints. We emphasize that the resulting SLS reformulation retains the locality structure of the original problem, which is key to enabling a distributed solution to the problem.

As previously stated, in the noisy case we restrict ourselves to polytopic constraints, such that  $$\mathcal P = \{[\mathbf{x}^\intercal \ \mathbf{u}^\intercal]^\intercal:\ H[\mathbf{x}^\intercal \ \mathbf{u}^\intercal]^\intercal\leq h\},$$ where $H := \emph{\textbf{blkdiag}}(H_{x,1},...,H_{x,T},H_{u,1},...,H_{u,T-1})$ and $h := (h_{x,1},...,h_{x,T},h_{u,1},...,h_{u,T-1})$. We consider two scenarios for the structure of the noise, which to distinguish from the initial condition $x_0$ we denote by $\boldsymbol{\delta}$, i.e., $\mathbf{w} = [x_{0}^\mathsf{T}\ \boldsymbol{\delta}^\mathsf{T}]^\mathsf{T}$.

\subsubsection{Locally norm-bounded disturbance}

\begin{definition}
$\mathcal W_\sigma$ is a \emph{separable local norm-bounded set} if for every signal $\boldsymbol \delta\in\mathcal W_\sigma$, $\| [\boldsymbol \delta]_i \|_p \leq \sigma$  $\forall i$ and $p\mathbb Z_{\geq 1}$.
\end{definition}

\begin{remark}
Note that $\boldsymbol \delta\in\mathcal W_\sigma$ if and only if $\| \boldsymbol \delta \|_\infty \leq \sigma$.
\end{remark}

\begin{lemma}\label{lemm:dual-norm}
Let the noise signal belong to a separable \emph{local} norm-bounded set $\mathcal W_\sigma$. Then, problem \eqref{eqn:MPC-SLS} with additional localization constraints has the following convex reformulation
\begin{align}
& \underset{\mathbf{\Phi}}{\text{min}} & f(\mathbf{\Phi}\{1\}x_{0})  \label{eqn:MPC-SLS_dual_norm} & ~\\
&  \text{s.t.} &\hspace{-2mm} Z_{AB}\mathbf{\Phi}=I , ~ & x_0  = x(\tau), ~ \mathbf{\Phi}\in\mathcal{L}_d, ~  [H\mathbf{\Phi}\{1\}]_i [x_0]_i \nonumber \\
& ~ & ~ & +  \sum_j\sigma \left\Vert e_j^\intercal [H \mathbf{\Phi}\{2:T\}]_i \right\Vert_*e_j \leq[h]_i \nonumber
\end{align}
where $e_j$ are vectors of the standard basis, and $\left\Vert \cdot \right\Vert_*$ the dual norm of $\left\Vert \cdot \right\Vert_p$.
\end{lemma}

\begin{proof}
The proof follows from a simple duality argument on the robust polytopic constraint $\mathbf \Phi \mathbf w \in \mathcal P\;\forall \mathbf w \in\mathcal W$. In particular, $\mathbf \Phi$ has to satisfy $$H\mathbf{\Phi}\{1\} x_0 + \max_{\{\left\Vert [\boldsymbol{\delta}]_i \right\Vert_p\ \forall i \} } H\mathbf \Phi\{2:T\} \boldsymbol{\delta} \leq h,$$ where the maximization over $\boldsymbol{\delta}$ and the inequality are element-wise. Since $\mathbf\Phi$ is localized and the disturbance set is separable and local, this can be equivalently written as: $$[H\mathbf{\Phi}\{1\}]_i [x_0]_i+\max_{\left\Vert [\boldsymbol{\delta}]]_i \right\Vert_p} [H\mathbf \Phi\{2:T\}]_i [\boldsymbol{\delta}]_i \leq [h]_i \ \forall i.$$ And by definition of dual norm gives, the second term on the left-hand side gives $$\sum_j\sigma \left\Vert e_j^\intercal [H\mathbf \Phi\{2:T\}]_i \right\Vert_* e_j \leq[h]_i  \ \forall i.$$ \end{proof}

\begin{remark}
Notice that the argument of Lemma \ref{lemm:dual-norm} can be extended to any setting for which a closed-form expression for $\sum_j \Big( \sup_{[\mathbf{w}]_i \in \mathcal{W}^i} e_j^\intercal [H \mathbf{\Phi}]_i [\mathbf{w}]_i \Big)e_j\leq[h]_i$ can be computed. Moreover, since $H$ is localized block-diagonal, and $\mathbf\Phi$ is localized, this constraint can be enforced using only local information.
\end{remark}

\subsubsection{Polytopic disturbance} 

\begin{lemma} \label{lemm:dual-polytope}
Let the noise signal belong to a polytope, i.e., $\boldsymbol \delta \in\{\boldsymbol \delta :\ G\boldsymbol \delta \leq g\}$, where $$G := \emph{\textbf{blkdiag}}(G_{1},...,G_{T}) \text{ and } g := (g_{1},...,g_{T}),$$ and each of the $\{G_t\}_{t=1}^{T}$ are block-diagonal with structure compatible with subsystem-wise decomposition of $\mathbf{w}$. Then, problem \eqref{eqn:MPC-SLS} localized can be reformulated as
\begin{align}
& \underset{\mathbf{\Phi},\mathbf{\Xi}\geq0}{\text{min}} & f(\mathbf{\Phi}\{1\}x_{0})   \label{eqn:MPC-SLS_dual_polytopic} & ~\\
&  \text{s.t.} & Z_{AB}\mathbf{\Phi}=I, & ~ x_0 = x(\tau), ~ \mathbf{\Phi}\in\mathcal{L}_d, ~ \mathbf{\Xi}\in\mathcal{L}_{d_H}, \nonumber \\
& ~ &H\mathbf{\Phi}\{1\}x_{0} + & \mathbf{\Xi}g \leq h,\ H\mathbf{\Phi}\{2:T\}=\mathbf{\Xi}G, \nonumber
\end{align}
with the constraint $\mathbf{\Xi}\geq0$  satisfied component-wise, and $\mathcal{L}_{d_H}$ denotes the subspace of matrices with the same sparsity as $H\mathbf{\Phi}\{2:T\}$. 
\end{lemma}

\begin{proof}
The reformulation of the robust polytopic constraint follows from duality. In particular, $\mathbf \Phi$ has to satisfy $$H\mathbf{\Phi}\{1\} x_0 + \max_{G\boldsymbol\delta\leq g} H\mathbf \Phi \{2:T\} \boldsymbol{\delta} \leq h.$$ We study this constraint row-wise and focus on the second term on the left-hand side, which can be seen as a linear program (LP) in $\boldsymbol\delta$ for each row of $H$. Since strong duality holds, we have that each of the LPs can equivalently be replaced by its corresponding dual problem. In particular, we can solve for the $k^\text{th}$ row of the the second term on the left-hand side as
\begin{align*}
&\underset{\mathbf\Xi(k,:)\geq0}{\text{min}} &\mathbf\Xi(k,:) g &~\\
&\text{s.t.} &H(k,:)\mathbf\Phi\{2   & :   T\}=\mathbf\Xi(k,:) G, 
\end{align*}
where the matrix operator $\mathbf\Xi$ results from stacking the dual variables from the row-wise dual LPs. Hence, the robust polytopic constraint can be replaced by $$ H\mathbf{\Phi}\{1\}x_{0} + \mathbf{\Xi}g \leq h,\ H\mathbf{\Phi}\{2:T\}=\mathbf{\Xi}G, $$ where $\boldsymbol\Xi\geq0$ (satisfied component-wise) becomes a decision variable of the MPC problem. Furthermore, the constraint $H\mathbf\Phi\{2 : T\}=\mathbf\Xi G$ allows for a sparse structure on $\mathbf \Xi$. In particular, $\mathbf\Xi G$ has to have the same sparsity as $H\mathbf\Phi\{2 : T\}$. When $G$ is block-diagonal, it immediately follows that if $\sum_kH(i,k)\mathbf\Phi\{2 : T\}(k,j)=0$, then $\mathbf\Xi(i,j)=0$. Hence, $\mathbf{\Xi}$ lies in $\mathcal{L}_{d_H}$, and together with the dual reformulation of the robust polytopic constraint $G\boldsymbol\delta\leq g$, gives rise to problem \eqref{eqn:MPC-SLS_dual_polytopic}. \end{proof}

\begin{remark}
By Assumption \ref{assump: locality}, the subspace $\mathcal{L}_{d_H}$ contains matrices with sparsity such that subsystems at most $2d$-hops away are coupled. 
\end{remark}

These results allow us to solve the DLMPC problem \eqref{eqn:MPC-SLS} using standard convex optimization methods, and further preserves the locality structure of the original problem under the given assumptions.  Imposing $d$-local structure on the system responses, coupled with an assumption of compatible $d$-local structure on the objective functions and constraints of the MPC problem \eqref{eqn:MPC}, leads to a \emph{structured} SLS MPC optimization problem \eqref{eqn:MPC-SLS}.  This structural compatibility in \emph{all optimization variables, cost functions, and constraints} is the key feature that we exploit in Section \ref{sec:admm} to apply distributed optimization techniques to scalably and \emph{exactly} solve problem \eqref{eqn:MPC-SLS}.\footnote{Notice that these locality constraints are defined in terms of the topology $\mathcal{G}_{(A,B)}$, so the structure imposed on $\mathbf\Phi$ will be compatible with the structure of the matrix $Z_{AB}$ that defines constraint \eqref{eqn:Z_AB}.}

\textbf{Why previous methods are not amenable to distributed solutions:} While previous methods \cite{goulart_optimization_2006,furieri_robust_2017} allow for similar structural constraints to be imposed on the controller realization through the use of either disturbance feedback or Youla parameterizations (subject to Quadratic Invariance \cite{rotkowitz_characterization_2006} conditions), the resulting synthesis problems do not enjoy the structure needed for distributed optimization techniques to be effective. We focus on the method defined in \cite{goulart_optimization_2006}, as a similar argument applies to the synthesis problem in \cite{furieri_robust_2017}.  Intuitively, the disturbance based feedback parameterization of \cite{goulart_optimization_2006} only parameterizes the closed-loop map $\mathbf \Phi_u$ from $\mathbf w \to \mathbf u$, and leaves the state $\mathbf x$ as a free variable. Hence, regardless as to what structure is imposed on the objective functions, constraints, and the map $\mathbf{\Phi}_u$, the resulting optimization problem is strongly and globally coupled because the state variable $\mathbf{x}$ is always dense.  This can be made explicit by noticing that the disturbance feedback parameterization of \cite{goulart_optimization_2006} can be recovered from the SLS parameterization of Theorem \ref{thm: SLS}, and we present a formal derivation in the Appendix. A similar coupling arises in the Youla based parameterization suggested in \cite{furieri_robust_2017}. This limits  their usefulness to smaller scale examples where centralized computation of policies is feasible. In contrast, by explicitly parameterizing the additional system response $\mathbf{\Phi}_x$ from $\mathbf{w}\to\mathbf{x}$, we can naturally enforce the structure needed for distributed optimization techniques to be fruitfully applied.  

\section{Distributed AND Localized MPC based on ADMM}\label{sec:admm}

We now show that ADMM \cite{boyd_distributed_2010} can be used to distribute the DLMPC problem \eqref{eqn:MPC-SLS_dual}. We exploit the separability (Assumption \ref{assump: locality}), locality constraints, and the notion of row/column-wise separability (to be defined next), to solve each of the local subproblems in parallel and with $d$-local information only. In what follows we restrict Assumption \ref{assump: locality} to the case where only dynamical coupling is considered. In our companion paper \cite{amoalonso_guarantees_2021}we extend these results to all cases considered in Assumption \ref{assump: locality} i.e., constraints and objective functions that introduce $d$-localized couplings. Hence, all cost function and constraints have structure:
\begin{gather*}
f(\mathbf x,\mathbf u) = \sum_{i = 1}^N f^i([\mathbf x]_i, [\mathbf u]_i), \text{ and} \\
[\mathbf x^\intercal\ \mathbf u^\intercal]^\intercal \in \mathcal P \text{ if and only if } \begin{bmatrix}[\mathbf x]_i \\ [\mathbf u]_i \end{bmatrix}\in\mathcal P^i,
\end{gather*}
where $\mathcal P = \mathcal P^1\times\dots\times\mathcal P^N$. 

By definition of the SLS system responses \eqref{eqn:Phis}, we can equivalently write these conditions in terms of $\mathbf \Phi$ as  
\begin{gather*}
f(\mathbf \Phi\{1\}x_0) = \sum_{i = 1}^N f^i\big(\mathbf \Phi\{1\}(\mathfrak{r}_i,:)x_0\big), \text{ and} \\
\mathbf \Phi \in \mathcal P \text{ if and only if } \mathbf\Phi(\mathfrak{r}_i,:)x_0\in\mathcal P^i\ \forall i,
\end{gather*}
where $\mathfrak r_i$ is the set of rows in $\mathbf \Phi$ corresponding to subsystem $i$, i.e., those that parametrize $[\mathbf x]_i$ and $[\mathbf u]_i$. These separability features are formalized as follows: 
\begin{definition}
Given the partition $\{\mathfrak r_1,...,\mathfrak r_k\}$, a functional/set is \emph{row-wise separable} if:
\begin{itemize}
\item For a functional, $g(\mathbf \Phi) = \sum_{i = 1}^k g_i\big(\mathbf \Phi(\mathfrak{r}_i,:)\big)$ for some functionals $g_i$ for $i=1,...,k$.
\item For a set, $\mathbf \Phi \in \mathcal P$ if and only if $\mathbf \Phi(\mathfrak{r}_i,:) \in \mathcal P_i \ \forall i$ for some sets $\mathcal P_i$ for $i=1,...,k$.
\end{itemize}
\end{definition}
An analogous definition exists for \emph{column-wise separable} functionals and sets \cite{wang_separable_2018,anderson_system_2019}, where the partition $\{\mathfrak c_1,...,\mathfrak c_k\}$ entails the columns of $\mathbf \Phi$, i.e., $\mathbf \Phi(:,\mathfrak{c}_i)$.

When the objective function and all the constraints of an optimization problem are separable with respect to a partition of cardinality $k$, then the optimization trivially decomposes into $k$ independent subproblems. However, this is not the case for the localized DLMPC problem \eqref{eqn:MPC-SLS}, since some elements are row-wise separable while others are column-wise. To make it amenable to a distributed solution, we propose the following reformulation, which applies to the noise-free case as well as to both noisy cases (locally norm-bounded and polytopic) considered in the previous section:
\begin{align}
& \underset{\mathbf{\tilde\Phi},\mathbf{\tilde\Psi}}{\text{min}} & f(M_1\mathbf{\tilde \Phi}\{1\}x_{0}) \label{eqn:MPC-SLS_dual} & ~\\
&  \text{s.t.} & Z_{AB}M_2\mathbf{\tilde \Psi}=I, & ~ x_0 = x(\tau), ~ \mathbf{\tilde\Phi},\mathbf{\tilde \Psi}\in\mathcal{L}_d, \nonumber \\
& ~ &\mathbf{\tilde\Phi}\in\mathcal{\tilde P},\ \mathbf{\tilde \Phi}&=\tilde H\mathbf{\tilde\Psi}, \nonumber
\end{align}

where: 
\begin{itemize}
\item In the noise-free case:
$$\hspace{-5mm}\mathbf{\tilde{\Phi}} := \mathbf{\Phi}\{1\}, ~  \mathbf{\tilde{\Psi}} := \mathbf{\Psi}\{1\}, ~ M_1=M_2=\tilde{H} =: I,$$
$$\tilde{\mathcal P}=\{\mathbf\Phi\{1\}:\ \mathbf\Phi\{1\}x_0\in\mathcal P\}.$$

\item In the noisy case:
$$M_1 := \begin{bmatrix} I & 0 \\ 0 & 0\end{bmatrix}, M_2 := \begin{bmatrix} 0 & I \end{bmatrix}, ~ \tilde{H} := \begin{bmatrix}I & 0 \\ 0 & H\end{bmatrix},$$

\begin{itemize}
\item When noise is locally bounded:
$$\hspace{-5mm}\mathbf{\tilde{\Phi}} := \begin{bmatrix} \mathbf{\Phi}\{1\} & \hspace{-2mm}0 \\ \mathbf{\Omega}\{1\} & \hspace{-2mm}\mathbf{\Omega}\{2:T\} \end{bmatrix}, ~ \mathbf{\tilde{\Psi}} := \begin{bmatrix} \mathbf{\Psi}\{1\} & \hspace{-2mm}0 \\ \mathbf{\Psi}\{1\} & \hspace{-2mm}\mathbf{\Psi}\{2:T\}\end{bmatrix},$$
$$\hspace{-5mm} \tilde{\mathcal P} := \{ \mathbf \Omega : [\mathbf{\Omega}\{1\}]_i [x_0]_i + \ \sum_j\sigma \left\Vert e_j^\intercal [\mathbf{\Omega}\{2:T\}]_i e_j\right\Vert_* \leq[h]_i \forall i \}.$$

\vspace{-3mm}
\item When noise is polytopic:
$$\hspace{-5mm}\mathbf{\tilde{\Phi}} := \begin{bmatrix} \mathbf{\Phi}\{1\} & 0 \\ \mathbf{\Omega}\{1\} & \mathbf{\Xi}G\end{bmatrix}, ~ \mathbf{\tilde{\Psi}} := \begin{bmatrix} \mathbf{\Psi}\{1\} & 0 \\ \mathbf{\Psi}\{1\} & \mathbf{\Psi}\{2:T\}\end{bmatrix}, $$
$$\hspace{-4mm}\tilde{\mathcal P} :=\{\mathbf\Omega,\mathbf\Xi : \mathbf{\Omega}\{1\}x_{0} + \mathbf{\Xi}g \leq h, ~ \mathbf\Xi\geq0 \}.$$
\end{itemize}
\end{itemize}

The matrices $\mathbf \Psi$ and $\mathbf \Omega$ are simply duplicates of $\mathbf \Phi$, and $\mathbf \Xi$ is the dual variable as introduced in Lemma \ref{lemm:dual-polytope}. The advantage of introducing these variables and creating the augmented variables $\tilde{\mathbf \Phi}$ and $\tilde{\mathbf \Psi}$ is that all components of problem \eqref{eqn:MPC-SLS_dual} involving $\tilde{\mathbf \Phi}$ are row-wise separable, and all components involving $\tilde{\mathbf \Psi}$ are column-wise separable. We can easily separate these two computations via ADMM using the relationship between both variables $\tilde H\mathbf{\tilde \Phi}=\mathbf{\tilde\Psi}$. Furthermore, we take advantage of the structure of these subproblems and separate them with respect to a row and column partition induced by the subsystem-wise partitions of the state and control inputs, $[\mathbf x]_i$ and $[\mathbf u]_i$ for each subsystem $i$. Each of these row and column subproblems resulting from the distribution across subsystems can be solved independently and in parallel, where each subsystem solves for its corresponding row and column partition. Moreover, since locality constraints are imposed, the decision variables $\tilde{\mathbf\Phi},\tilde{\mathbf\Psi}$ have a sparse structure. This means that the length of the rows and columns that a subsystem solves for is much smaller than the length of the rows and columns of $\mathbf{\Phi}$. For instance, when considering the column-wise subproblem evaluated at subsystem $i$, the $j^\text{th}$ row of the $i^\text{th}$ column partitions of $\mathbf\Phi_x$ and $\mathbf\Phi_u$ is nonzero only if $j\in\cup_{k\in\textbf{out}_i(d)}\mathfrak{r}_k$ and $j\in\cup_{k\in\textbf{out}_i(d+1)}\mathfrak{r}_k$ respectively.\footnote{An algorithm to find the relevant components for each subsystem rows and columns can be found in Appendix A of \cite{wang_separable_2018}.} Thus, the subproblems that subsystem $i$ solves for are:
\begin{subequations}\label{eqn:DLMPC}
{\allowdisplaybreaks 
\begin{align}
	&\hspace{-1.5mm}[\mathbf{\tilde{\Phi}}]_{i_r}^{k+1}  \hspace{-1mm} = \hspace{-1mm}
		\left\{\begin{aligned}
			&\hspace{-.5mm}\underset{[\mathbf{\tilde\Phi}]_{i_r}}{\text{argmin}} & \hspace{-2mm}f([M_1\mathbf\Phi]_{i_r}[x_0]_{i_r}) \hspace{-1mm}+\hspace{-1mm} \frac{\rho}{2} g_{i_r}( \mathbf{\tilde\Phi},\mathbf{\tilde\Psi}^k,\mathbf{\Lambda}^k)&\\
			&~~\text{s.t.}  &\hspace{-3.5mm}[\mathbf{\tilde\Phi}]_{i_r}\in\tilde{\mathcal{P}}^i\cap\mathcal L_d,~[x_0]_{i_r}=[x(\tau)]_{i_r}&
		\end{aligned}\right\} \label{eqn:DLMPC-row}\ 
	\\ \vspace{5pt}
	&\hspace{-1.5mm} [\mathbf{\tilde\Psi}]_{i_c}^{k+1} = 
	\left\{\begin{aligned}
			&\underset{[\mathbf{\tilde\Psi}]_{i_c}}{\text{argmin}} &g_{i_c}( \mathbf{\tilde\Phi}^{k+1},\mathbf{\tilde\Psi},\mathbf{\Lambda}^k)&\\
			&~~\text{s.t.}  & [Z_{AB}]_{i_c}[M_2\mathbf{\tilde\Psi}]_{i_c} = & [I]_{i_c}	
	\end{aligned}\right\} \label{eqn:DLMPC-column}
	\\[5pt]
	  &\hspace{-1.5mm} [\mathbf{\Lambda}]_{i_c}^{k}  =  g_{i_c}( \mathbf{\tilde\Phi}^{k+1},\mathbf{\tilde\Psi}^{k+1},\mathbf{\Lambda}), \label{eqn:DLMPC-lagrange}
\end{align}}
\end{subequations}  
where the scalar $\rho$ is the ADMM multiplier, operator $\mathbf \Lambda$ is the dual variable associated with ADMM, and
$$g_{*}( \mathbf{\tilde\Phi},\mathbf{\tilde\Psi},\mathbf{\Lambda}) = \left\Vert [\mathbf{\tilde\Phi}]_{*} - [\tilde H]_{*} [\mathbf{\tilde\Psi}]_{*} + [\mathbf{\Lambda}]_{*} \right\Vert_F^2,$$
where $*$ indicates ${i_r}$ or ${i_c}$ depending on the context. To simplify notation we denote as $[\mathbf\Phi]_{i_r}$ the submatrix of $\mathbf\Phi$ formed by the nonzero components of the relevant rows for subsystem $i$, $\mathfrak{r}_i$, and as $[\mathbf\Phi]_{i_c}$ the submatrix of $\mathbf\Phi$ formed by the nonzero components of the relevant columns for subsystem $i$, $\mathfrak{c}_i$. We use a similar bracket notation for the vectors and matrices that multiply the decision variables to indicate that those are composed from the relevant components of their global versions. 

In subroutine \eqref{eqn:DLMPC}, the computation of iterate \eqref{eqn:DLMPC-column} can be sped up by virtue of Lemma \ref{lem:closed-form} (Appendix) using the closed-from solution 
\begin{equation}\label{eqn:closed_form}
	[\mathbf{\Psi}]_{i_c}^{k+1} = z^{*} 
	\Big( [\tilde H]_{i_c}, [\mathbf{\tilde\Phi}]^{k+1}_{i_c}+[\mathbf{\tilde\Lambda}]^k_{i_c}, [Z_{AB}M_2]_{i_c}, [I]_{i_c} \Big).
\end{equation}

Notice that the number of nonzero components for both the rows and the columns is much smaller than the size of the network $N$ since it is determined by the size of the local neighborhood $d$ through the locality constraints. In turn, this implies that the subsystem only requires small submatrices of the system and constraint matrices $A,B,H,$ etc, to perform the computations. Therefore, the DLMPC subroutine \eqref{eqn:DLMPC} can be solved in parallel in a distributed manner across the subsystems of the network, where each solves for a local patch of the system responses using local information only. In Algorithm \ref{alg:DLMPC} we present the DLMPC algorithm that each sub-controller executes, and where only local information exchanges (within $d$-hop neighbors) take place.

\setlength{\textfloatsep}{0pt}
\begin{algorithm}[ht]
\caption{Subsystem $i$ DLMPC implementation}\label{alg:DLMPC}
\begin{algorithmic}[1]
\Statex \textbf{input:} $\epsilon_p, \epsilon_d, \rho>0$.
\State Measure local state $[x(\tau)]_{i}$, $k\leftarrow0$.
\State Share the measurement with neighbors in $\textbf{out}_{i}(d)$.  Receive the corresponding $[x(\tau)]_j$ from $\textbf{in}_{i}(d)$ and build $[x_0]_{i_r}$.
\State Solve optimization problem \eqref{eqn:DLMPC-row}.
\State Share $[{\mathbf{\tilde\Phi}}]_{i_{r}}^{k+1}$ with $\textbf{out}_{i}(d)$. Receive the corresponding $[{\mathbf{\tilde\Phi}}]_{j_{r}}^{k+1}$ from $\textbf{in}_{i}(d)$ and build $[{\mathbf{\tilde\Phi}}]_{i_{c}}^{k+1}$.
\State Solve problem \eqref{eqn:DLMPC-column} via the closed form \eqref{eqn:closed_form}.
\State Share $[{\mathbf{\tilde\Psi}}]_{i_{c}}^{k+1}$ with $\textbf{out}_{i}(d)$. Receive the corresponding $[{\mathbf{\tilde\Psi}}]_{j_{c}}^{k+1}$ from $\textbf{in}_{i}(d)$ and build $[{\mathbf{\tilde\Psi}}]_{i_{r}}^{k+1}$.
\State Perform the multiplier update step \eqref{eqn:DLMPC-lagrange}.
\State \textbf{if} {$\left\Vert[{\mathbf{\tilde\Phi}}]_{i_{r}}^{k+1}-[\tilde{H}]_{i_{r}}[{\mathbf{\tilde\Psi}}]_{i_{r}}^{k+1}\right\Vert_F\leq\epsilon_{p}$
\Statex  and $\left\Vert[{\mathbf{\tilde\Psi}}]_{i_{r}}^{k+1}-[{\mathbf{\tilde\Psi}}]_{i_{r}}^{k}\right\Vert_F\leq\epsilon_{d}$}\textbf{:}
\Statex $\;\;$ Apply control action $[u_0]_i = [\Phi_{u,0}[0]]_{i_{r}}[x_0]_{i_r}$, and return to step 1.
\Statex \textbf{else:} 
\Statex $\;\;$ Set $k\leftarrow k+1$, return to step 3.
\end{algorithmic}
\end{algorithm}

\textbf{Computational complexity of Algorithm \ref{alg:DLMPC}:}         
The complexity of the algorithm is determined by update steps 3, 5 and 7. In particular, steps 5 and 7 can be directly solved in closed form, reducing their evaluation to the multiplication of matrices of dimension $O(d^2T^2)$ in the noisy case, and $O(d^2T)$ in the noise-free case. In general, step 3 requires an optimization solver where each local iterate sub-problem is over $O(d^{2}T^2)$ optimization variables in the noisy case, and $O(d^2T)$ in the noise-free case, subject to $O(dT)$ constraints. In certain cases, step 3 can also be computed in closed form if a proximal operator exists for the formulation \cite{boyd_distributed_2010}. For instance this is true if it reduces to a quadratic convex cost function subject to affine equality constraints, in which case complexity reduces to $O(d^{2}T)$ since piece-wise closed form solutions can be computed \cite{amo_alonso_explicit_2020}. Notice that the complexity of the subproblems is largely dependent on whether noise is considered. The reason for this is that in the noise-free case only the first block-column of $\mathbf\Phi$ is considered, whereas in the presence of noise all $T$ block-columns have to be computed. Regardless, the use of locality constraints leads to a significant computational saving when $d\ll n$. 
The communication complexity - as determined by steps 2, 4 and 6 - is limited to the local exchange of information between $d$-local neighbors. 

\section{Simulation Experiments} \label{simulation}

We now apply the DLMPC algorithm to a power system inspired example. After introducing the simplified model, we present simulations under different noise realizations and validate the algorithm correctness and optimal performance by comparing to a centralized algorithm, as well as its ability to achieve constraint satisfaction in the presence of noise. We further demonstrate the scalability of the proposed method by verifying different network and problem parameters (locality, network size and time horizon), and in particular show that runtime stays steady as network size increases. Code needed to replicate these experiments is available at \texttt{\url{https://github.com/unstable-zeros/dl-mpc-sls}}; this code makes use of the SLS toolbox \cite{slstoolbox}, which includes ready-to-use MATLAB implementations of all algorithms presented in this paper and its companion paper \cite{amoalonso_guarantees_2021}.

\subsection{System model}

We begin with a two-dimensional square mesh, where we randomly determine whether each node connects to each of its neighbors with a $40\%$ probability. The expected number of edges is $ 0.8*n*(n-1)$. Each node represents a two-state subsystem that follows linearized and discretized swing dynamics
\begin{equation*}
\begin{bmatrix} \theta(t+1) \\ \omega(t+1) \end{bmatrix}_i = \sum_{j\in\textbf{in}_i(1)}[A]_{ij} \begin{bmatrix} \theta(t) \\ \omega(t) \end{bmatrix}_j + [B]_{i}[u]_i + [w]_i,
\end{equation*}
where $[\theta]_i$, $[\dot{\theta}]_i$, $[u]_i$ are the phase angle deviation, frequency deviation, and control action of the controllable load of bus $i$. The dynamic matrices are 
\[[A]_{ii}=\begin{bmatrix}
   1  & \Delta t \\
  -\frac{k_i}{m_i}\Delta t &  1-\frac{d_i}{m_i}\Delta t
\end{bmatrix}, \ [A]_{ij}=\begin{bmatrix}
   0  & 0 \\
   \frac{k_{ij}}{m_i}\Delta t &  0
\end{bmatrix},\] 
and $[B]_{ii}=\begin{bmatrix} 0 & 1 \end{bmatrix}^\intercal$ for all $i$. Parameters in bus $i$: $m_i^{-1}$ (inertia inverse), $d_i$ (damping) and $k_{ij}$ (coupling) are randomly generated and uniformly distributed between $[0,\ 2]$, $[0.5,\ 1]$, and $[1,\ 1.5]$ respectively. The discretization step is $\Delta t = 0.2$, and $k_i :=  \sum_{j\in\textbf{in}_i(1)} k_{ij}$.

In simulations, we start with a randomly-generated initial condition, and we study three noise scenarios: noise-free, polytopic noise, and locally-bounded noise. Noise follows a uniform distribution, and in the locally-bounded case it is scaled appropriately to meet the local bounds at each patch (see code for details). The baseline parameter values are $d=3,\ T=5,\ N=16$ ($4\times4$ grid). 

\subsection{Optimal performance}

We observe the trajectories of the closed-loop system when using DLMPC for the different noise realizations. We compare these results with the solution to the corresponding centralized MPC problem (using CVX \cite{cvx}). For a randomly chosen initial condition and network topology, we plot the evolution of the two states of subsystem $4$ under the different noise conditions. We use a quadratic cost and polytopic constraints on both angle and frequency, and impose upper and lower bounds.

\begin{figure}[h]
\centering
\includegraphics[width=.7\columnwidth]{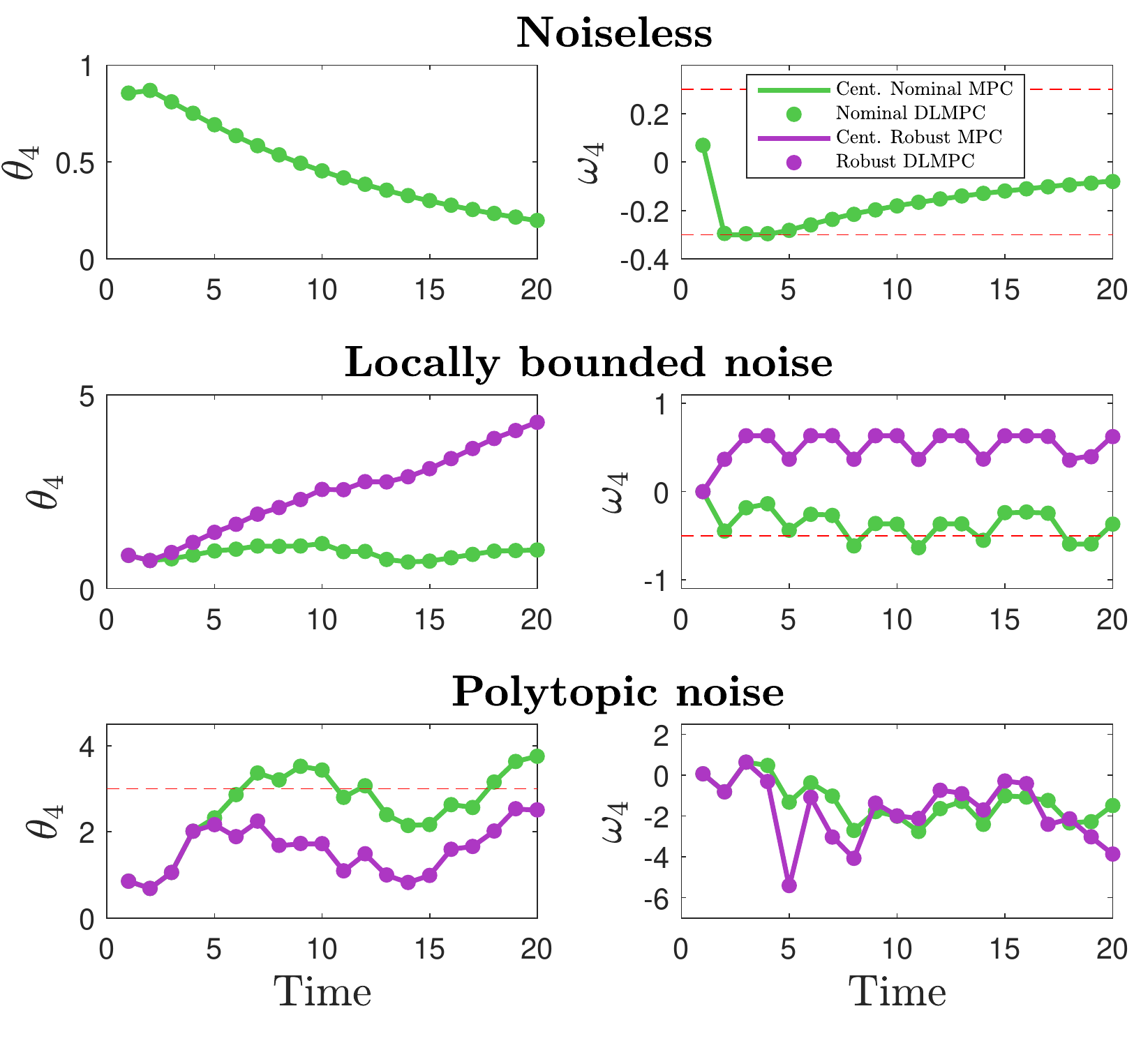}
\caption{Evolution of the states of subsystem $4$ under a nominal (green) and a robust (purple) controller. Red dashed line indicates upper or lower bound respectively. Both DLMPC and centralized MPC controller yield the same result. In the absence of noise, both the nominal and robust controller lead to the same trajectory, whereas in the presence of noise (both locally-bounded and polytopic), the nominal controller leads to bounds violation.}
\label{fig:dynamics}
\end{figure}

Results from simulations are summarized in Fig. \ref{fig:dynamics}. In all cases, the centralized solution coincides with the solution achieved by the DLMPC algorithm, validating the optimality of the algorithm proposed. It is worth noting that in the absence of noise, both nominal and robust DLMPC yield the same result, illustrating that the robust formulation of DLMPC is a generalization of the nominal case, but the former is more computationally efficient. In the noisy cases, the nominal DLMPC leads to bound violations, illustrating the need for a robust approach. It is worth noting the conservativeness of the algorithm in the locally-bounded case is much more apparent than in the polytopic noise case. The reason for this is that in general there does not exist a noise realization that is worst-case for \emph{every} local patch at the same time. This is in contrast with polytopic noise where there does exist a worst-case noise. Therefore locally-bounded noise formulations inherently introduce some conservativeness as they anticipate a worst-case that is not mathematically possible in general. This is because it assumes a worst-case local disturbance at each patch; however, patches overlap, and the worst-case disturbance at some node $i$ is different for each patch containing $i$.

\subsection{Computational complexity}

To asses the scalability of the algorithm, we measure runtime\footnote{Runtime is measured after the first iteration, so that all the iterations for which runtime is measured are warmstarted.} while varying different network and problem parameters: locality $d$, network size $N$\footnote{To increase network size, we vary the size of the grid over $4\times4$ ($32$ states), $6\times6$ ($72$ states), $8\times8$ ($128$ states), and $11\times11$ ($242$ states) grid sizes.} and time horizon $T$. We run $5$ different simulations for each of the parameter combinations, using different realizations of the randomly chosen parameters as to provide a consistent estimation of runtimes. 

We study the scalability of the DLMPC algorithm for each of the different computation strategies presented: noiseless, locally-bounded noise, and polytopic noise. For the sake of comparison, we include an additional computation strategy based on explicit MPC that reduces the overhead in the noiseless case by replacing the optimization solver by a piece-wise solution. Interested readers are referred to \cite{amo_alonso_explicit_2020} for details. We observe the behavior for these four different strategies in Fig. \ref{fig:scalability}, and we note that the computation strategy determines the order of magnitude of the runtime, ranging from $10^{-3}-10^{-2}s$ in the noiseless case -- depending on whether or not explicit solutions are used -- to $10^{-1}s$ in the locally-bounded case, and around $1-10s$ when polytopic noise is considered. This difference is expected, and is explained by the size of the decision variables across different scenarios, i.e. $\mathbf{\tilde{\Phi}}$ has a much larger dimension in the polytopic case than in the noiseless case. Despite this difference in order of magnitude, the trends observed are the same for each of the different scenarios. In particular, runtime barely increases with the size of the network, and the slight increase in runtime -- likely due to the larger coordination needed -- does not seem to be significant and appears to level off for sufficiently large networks. These observations are consistent with those of \cite{conte_computational_2012} where the same trend was noted. In contrast, runtime appears to increase with time horizon, and more notably with locality region size. This is also expected, as according to our                                                                                                                                                                                                                                                                                                                                                                                                                                                                                                                                                                                                                                                                                                                                                                                                                                                                                                                                                                                                                                                                                                                                                                                                                                                                                                                                                                                                                                                                                                                                                                                                                                                                                                                                                                                                                                                                                                                                                                                                                                                                                                                                                                                                                                                                                                                                                                                                                                                                                                                                                                                                                                                                                                                                                                                                                                                                                                                                                                                                                                                                                                                                                                                                                                                                                                                                                                                                                                                                                                                                                                                                                                                                                                                                                      complexity analysis the number of variables in a DLMPC subproblem scales as $O(d^2T^2)$ for the robust cases and as $O(d^2T)$ for the nominal case. It is well-known that a larger time horizon, while providing an improvement in performance, can be detrimental for computational overhead. The same is true for locality size, which is much smaller than the size of the network. Although a larger localized region $d$ can lead to improved performance, as a broader set of subsystems can coordinate their actions directly, it also leads to an increase in computational complexity and communication overhead. Thus by choosing the smallest localization parameter $d$ such that acceptable performance is achieved, the designer can tradeoff between computational complexity and closed-loop performance in a principled way. This further highlights the importance of exploiting the underlying structure of the dynamics, which allow us to enforce locality constraints on the system responses, and consequently, on the controller implementation. 

\begin{figure}[htp]
    \centering
    \includegraphics[width=.7\columnwidth]{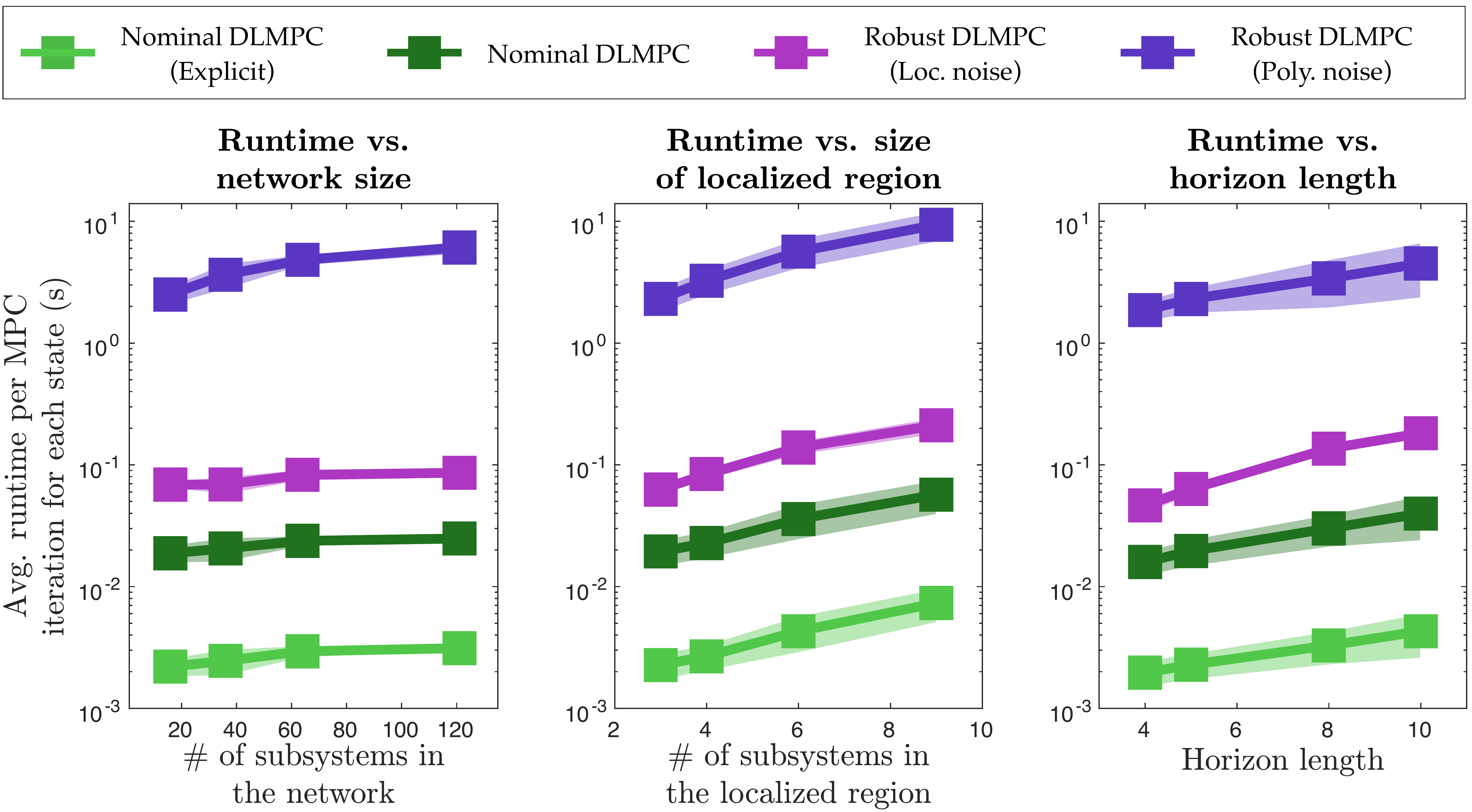}
   \caption{Average runtime per DLMPC iteration with network size (left), locality parameter (middle), and time horizon (right). The lines are the mean values and the shaded areas show the values within one standard deviation. Since computations are parallelized across subsystems, runtime is measured on a subsystem, normalized per state, and averaged out after the MPC controller is finished.}
    \label{fig:scalability}
\end{figure}

\section{Conclusion}

We defined and analyzed a \emph{closed-loop} Distributed and Localized MPC algorithm. By leveraging the SLS framework, we were able to enforce information exchange constraints by imposing locality constraints on the system responses.  We further showed that when locality is combined with mild assumptions on the separability structure of the objective functions and constraints of the problem, an ADMM based solution to the DLMPC subproblems can be implemented that requires only local information exchange and system models, making the approach suitable for large-scale distributed systems. This is the first DMPC algorithm that allows for the distributed synthesis of closed-loop policies. In our companion paper \cite{amoalonso_guarantees_2021}, we provide recursive feasibility and asymptotic stability guarantees for our approach, as well as scalable synthesis and implementations for the terminal constraint set and cost. 

Given that the presented algorithm relies on multiple exchanges of information between the subsystems, how communication loss affects the closed-loop performance of the algorithm is an interesting question. Although a formal analysis is left as future research, the work done in \cite{li_robust_2018} suggests that it would be possible to slightly modify the proposed ADMM-based scheme to make it robust to unreliable communication links. Finally, it is of interest to extend these results to information exchange topologies defined in terms of both sparsity and delays -- while the SLS framework naturally allows for delay to be imposed on the implementation structure of a distributed controller, it is less clear how to incorporate such constraints in a distributed optimization scheme.


\section*{Appendix}

\subsection{SLS based MPC and disturbance based feedback parametrizations}

\begin{proposition}
The disturbance based parametrization defined in Section 4 of \cite{goulart_optimization_2006} is a special case of the SLS parametrization \eqref{eqn:Phis}. 
\end{proposition}

\begin{proof}
We start with the statement from Theorem \ref{thm: SLS} and in particular affine constraint  \eqref{eqn:Z_AB}. Multiplying this constraint by $\mathbf{w}$ on the right we obtain:
\begin{equation}\nonumber
        \left[I-Z\hat A\ \ -Z\hat B\right]\left[\begin{array}{c}\mathbf{\Phi}_{x}\\\mathbf{\Phi}_{u}\end{array}\right] \mathbf{w} = \mathbf{w},
\end{equation}
where by definition of $\mathbf\Phi_x$, $\mathbf{x} =: \mathbf\Phi_x \mathbf w$. Hence, by the result of \cite{anderson_system_2019} the SLS MPC subproblem
\begin{equation}\nonumber
\begin{array}{rl}
\underset{\mathbf{x},\mathbf{\Phi}_{u}}{\text{min}} & f(\mathbf{x},\mathbf{\Phi}_{u}\mathbf{w})\\
 \text{s.t.} &  (I - Z\hat{A})\mathbf{x}=Z\hat{B}\mathbf{\Phi}_u \mathbf{w} + \mathbf{w},\, x_0 = x(t),\\
 &\mathbf{x}\in \mathcal{X}, \mathbf{\Phi}_{u}\mathbf{w} \in\mathcal{U} ~\forall \mathbf w \in \mathcal W,
\end{array}
\end{equation}
is equivalent to problem \eqref{eqn:MPC} when restricted to solving over linear-time-varying feedback policies. 

Now, if we consider nominal (disturbance free) cost, i.e. $\mathbf{w} = [x_0^\top,0,\dots,0]^\top$ the problem becomes
\begin{equation}\nonumber
\begin{array}{rl}
\underset{\mathbf{x},\mathbf{\Phi}_{u}}{\text{min}} & f(\mathbf{x},\mathbf{\Phi}_{u}\{1\}x_{0})\\
 \text{s.t.} &  (I - Z\hat{A})\mathbf{x}=Z\hat{B}\mathbf{\Phi}_u\mathbf w + \mathbf w,\, x_0 = x(t),\\
 &\mathbf{x}\in \mathcal{X}, \mathbf{\Phi}_{u}\mathbf w \in\mathcal{U} ~\forall \mathbf w \in \mathcal W,
\end{array}
\end{equation}
Notice that by setting $\mathbf M = \mathbf\Phi_u$ and $\mathbf v =\mathbf w$ , we recover the optimization problem over disturbance feedback policies suggested in Section 4 of \cite{goulart_optimization_2006}.  \end{proof}

An equivalent derivation arises in the Youla based parameterization suggested in \cite{furieri_robust_2017}.  

\subsection{Closed-form solutions for ADMM iterates}

\begin{lemma}\label{lem:closed-form}
Let $z^{\star}(M,v,P,q) := \underset{z}{\text{argmin}}\left\Vert Mz-v \right\Vert_F^2$ s.t. $Pz=q$. Then
$$\begin{bmatrix} z^{\star} \\ \mu^{\star} \end{bmatrix} = \begin{bmatrix}MM^\intercal & P^\intercal \\ P & 0\end{bmatrix}^\dagger \begin{bmatrix} M^\intercal v \\ q \end{bmatrix},$$
where  $\dagger$ denotes pseudo-inverse, is the optimal solution and $\mu^{\star}$ is the corresponding optimal Lagrange multiplier.
\end{lemma}
\begin{proof}
The proof follows directly from applying the KKT conditions to the optimization problem. By the stationarity condition, $M^\intercal M z^{\star}-M^\intercal v + P^\intercal \mu^{\star} = 0$, where $z^{\star}$ is the solution to the optimization problem and $\mu^{\star}$ the optimal Lagrange multiplier vector. From the primal feasibility condition, $Pz^{\star} = q$. Hence, $z^{\star}$ and $\mu^{\star}$ are computed as the solution to this system of two equations. 
\end{proof}

\section*{Acknowledgements}
Carmen Amo Alonso is partially supported by a Caltech/Amazon AI4Science fellowship. Jing Shuang (Lisa) Li is supported in part by a postgraduate scholarship from the Natural Sciences and Engineering Research Council of Canada [NSERC PGSD3-557385-2021]. James Anderson is partially supported by NSF CAREER ECCS-2144634 and DOE award DE-SC0022234. Nikolai Matni is supported in part by NSF awards CPS-2038873 and CAREER award ECCS-2045834, and a Google Research Scholar award.

\bibliographystyle{IEEEtran}
\bibliography{references}

\end{document}